\documentclass{amsart}

\usepackage{amssymb}
\usepackage[english]{babel}
\usepackage[T1]{fontenc}
\usepackage[utf8]{inputenc}
\usepackage{mathtools}
\usepackage{stmaryrd}
\usepackage{tikz-cd}

\numberwithin{equation}{section}

\newtheorem{lemm}[equation]{Lemma}
\newtheorem{prop}[equation]{Proposition}
\newtheorem{thrm}[equation]{Theorem}

\DeclareMathOperator{\Supp}{Supp}
\DeclareMathOperator{\val}{v}

\title{Pseudovaluations on the de Rham-Witt complex}
\author{Rubén \textsc{Muñoz-{}-Bertrand}}

\begin{document}

\begin{abstract}
For a polynomial ring over a commutative ring of positive characteristic, we define on the associated de Rham-Witt complex a set of functions, and show that they are pseudovaluations in the sense of Davis, Langer and Zink. To achieve it, we explicitly compute products of basic elements on the complex. We also prove that the overconvergent de Rham-Witt complex can be recovered using these pseudovaluations.
\end{abstract}

\maketitle

\section*{Introduction}

Davis, Langer and Zink have introduced the overconvergent de Rham-Witt complex in \cite{overconvergentderhamwittcohomology}. It is a complex of sheaves defined on any smooth variety $X$ over a perfect field $k$ of positive characteristic. It can be used to compute both the Monksy-Washnitzer and the rigid cohomology of the variety.

This complex is defined as a differential graded algebra (dga) contained in the de Rham-Witt complex $W\Omega_{X/k}$ of Deligne and Illusie. In order to achieve this they have defined for any $\varepsilon>0$, in the case where $X$ is the spectrum of a polynomial ring $k[\underline{X}]$ over $k$, a function $\gamma_{\varepsilon}\colon W\Omega_{k[\underline{X}]/k}\to\mathbb{R}\cup\{+\infty,-\infty\}$. The overconvergent de Rham-Witt complex is the set of all $x\in W\Omega_{k[\underline{X}]/k}$ such that $\gamma_{\varepsilon}(x)\neq-\infty$ for some $\varepsilon>0$. In the general case, it is defined as the functional image of this set for a surjective morphism of smooth commutative algebras over $k$.

In degree zero (that is, for Witt vectors), these applications have nice properties. One of those is that they are pseudovaluations. We recall the definition: a \textbf{pseudovaluation} on a commutative ring $R$ is a function $v\colon R\to\mathbb{R}\cup\{+\infty,-\infty\}$ such that:
\begin{gather*}
v(0)=+\infty\text{,}\\
v(1)=0\text{,}\\
\forall r\in R,\ v(r)=v(-r)\text{,}\\
\forall r,s\in R,\ v(r+s)\geqslant\min\{v(r),v(s)\}\text{,}\\
\forall r,s\in R,\ \left(v(r)\neq-\infty\right)\wedge\left(v(s)\neq-\infty\right)\implies\left(v(rs)\geqslant v(r)+v(s)\right)\text{.}
\end{gather*}

Pseudovaluations and their behaviour have been studied in \cite{overconvergentwittvectors}. It appears that they form a convenient framework to study the overconvergence of recursive sequences. However, there are counterexamples showing that in positive degree, the applications $\gamma_{\varepsilon}$ are not pseudovaluations. This becomes an obstacle when one wants to study the local structure of the overconvergent de Rham-Witt complex, or when one tries to find an interpretation of $F$-isocrystals for the overconvergent de Rham-Witt complex following the work of Ertl \cite{ertlcomparison}.

In this paper, we define new applications $\zeta_{\varepsilon}\colon W\Omega_{k[\underline{X}]/k}\to\mathbb{R}\cup\{+\infty,-\infty\}$ for all $\varepsilon>0$ and prove that these are pseudovaluations. Moreover, we show that the set of all $x\in W\Omega_{k[\underline{X}]/k}$ such that $\zeta_{\varepsilon}(x)\neq-\infty$ for some $\varepsilon>0$ also define the overconvergent de Rham-Witt complex.

In order to do so, we recall in the first section the main results concerning the de Rham-Witt complex, especially in the case of a polynomial algebra. The second section, which is the most technical one, consists of computations of products of specific elements of the de Rham-Witt complex. The results are explicit, and proven in the case where $k$ is any commutative $\mathbb{Z}_{(p)}$-algebra. This enable us in the last section to define the pseudovaluations, and prove that in the case of a perfect field of positive characteristic we retrieve with these functions the overconvergent de Rham-Witt complex.

In subsequent papers, we will use these results in order to give a new interpretation of $F$-isocrystals and study the structure of the overconvergent de Rham-Witt complex. These results can also be found in my PhD thesis.

\section*{Acknowledgements}

This work is a slight generalization of a part of my PhD thesis. As such, I am very indebted to my advisor Daniel Caro, and to Laboratoire de Mathematiques Nicolas Oresme at Université de Caen Normandie. I would also like to thank all the members of my jury Andreas Langer, Tobias Schmidt, Andrea Pulita, Christine Huyghe and Jérôme Poineau for all their comments. I also had the luck to talk about these topics with Bernard Le Stum.

\section{The de Rham-Witt complex for a polynomial ring}

Let $p$ be a prime number. Let $k$ be a commutative $\mathbb{Z}_{(p)}$-algebra. Let $n\in\mathbb{N}$ and write $k[\underline{X}]\coloneqq k[X_{1},\ldots,X_{n}]$. We will first recall basic properties of the de Rham-Witt complex of $k[\underline{X}]$, denoted $\left(W\Omega_{k[\underline{X}]/k},d\right)$ (for an introduction, see \cite{complexedederhamwittetcohomologiecristalline} or \cite{derhamwittcohomologyforaproperandsmoothmorphism}). In degree zero, $W\Omega^{0}_{k[\underline{X}]/k}$ is isomorphic as a $W(k)$-algebra to $W(k[\underline{X}])$, the ring of Witt vectors over $k[\underline{X}]$.

There is a morphism of graded rings $F\colon W\Omega_{k[\underline{X}]/k}\to W\Omega_{k[\underline{X}]/k}$, a morphism of graded groups $V\colon W\Omega_{k[\underline{X}]/k}\to W\Omega_{k[\underline{X}]/k}$, as well as a morphism of monoids $[\bullet]\colon(k[\underline{X}],\times)\to(W(k[\underline{X}]),\times)$ such that:
\begin{gather}
\forall r\in k[\underline{X}],\ F([r])=[r^{p}]\text{,}\label{frobwittvect}\\
\forall x,y\in W\Omega_{k[\underline{X}]/k},\ V(xF(y))=V(x)y\text{,}\label{vxfyvxy}\\
\forall m\in\mathbb{N},\ \forall x\in W\Omega_{k[\underline{X}]/k},\ d(F^{m}(x))=p^{m}F^{m}(d(x))\text{,}\label{dfpfd}\\
\forall m\in\mathbb{N},\ \forall P\in W\Omega_{k[\underline{X}]/k},\ F^{m}(d([P]))=[P^{p^{m}-1}]d([P])\text{,}\label{fdttdt}\\
\begin{multlined}\forall i,j\in\mathbb{N},\ \forall x\in W\Omega^{i}_{k[\underline{X}]/k},\ \forall y\in W\Omega^{j}_{k[\underline{X}]/k},\\d(xy)=(-1)^{i}xd(y)+(-1)^{(i+1)j}yd(x)\text{,}\end{multlined}\label{dbasicprop}\\
\forall m\in\mathbb{N},\ \forall(x_{i})_{i\in\llbracket1,m\rrbracket}\in\left(W(k[\underline{X}])\right)^{m},\ d\left(\prod_{i=1}^{m}x_{i}\right)=\sum_{i=1}^{m}\left(\prod_{j\in\llbracket1,m\rrbracket\smallsetminus\{i\}}x_{j}\right)d(x_{i})\text{.}\label{dproducts}
\end{gather}

We are going to introduce basic elements on the de Rham-Witt complex, and recall how any de Rham-Witt differential on $k[\underline{X}]$ can be expressed as a series using these elements. We mostly follow \cite{derhamwittcohomologyforaproperandsmoothmorphism}.

A \textbf{weight function} is a mapping $a\colon\llbracket1,n\rrbracket\to\mathbb{N}\left[\frac{1}{p}\right]$; for all $i\in\llbracket1,n\rrbracket$, its values shall be written $a_{i}$. We define:
\begin{equation*}
\lvert a\rvert\coloneqq\sum_{i=1}^{n}a_{i}\text{.}
\end{equation*}

For any weight function $a$ and any $J\subset\llbracket1,n\rrbracket$, we denote by $a|_{J}$ the weight function which for all $i\in\llbracket1,n\rrbracket$ satisfies:
\begin{equation*}
a|_{J}(i)=\left\{\begin{array}{ll}a_{i}&\text{ if }i\in J\text{,}\\0&\text{ otherwise.}\end{array}\right.
\end{equation*}

The \textbf{support} of a weight function $a$ is the following set:
\begin{equation*}
\Supp(a)\coloneqq\{i\in\llbracket1,n\rrbracket\mid a_{i}\neq0\}\text{.}
\end{equation*}

A \textbf{partition} of a weight function $a$ is a subset of $\Supp(a)$. Its \textbf{size} is its cardinal. We will denote by $\mathcal{P}$ the set of all pairs $(a,I)$, where $a$ is a weight function and $I$ is a partition of $a$.

In all this paper, the $p$-adic valuation shall be denoted $\val_{p}$. We fix the following total order $\preceq$ on $\Supp(a)$:
\begin{multline*}
\forall i,i'\in\Supp(a),\ i\preceq i'\\\iff\left(\left(\val_{p}(a_{i})\leqslant\val_{p}(a_{i'})\right)\wedge\left(\left(\val_{p}(a_{i})=\val_{p}(a_{i'})\right)\implies\left(i\leqslant i'\right)\right)\right)\text{.}
\end{multline*}

This order depends on the choice of a weight function $a$, but no confusion will arise. We will denote by $\prec$ the associated strict total order, and we set $\min(a)\in\Supp(a)$ the only element such that $\min(a)\preceq i$ for any $i\in\Supp(a)$.

Let $I\coloneqq\left\{i_{j}\right\}_{j\in\llbracket1,m\rrbracket}$ be a partition of a weight function $a$. We will always suppose that $i_{j}\prec i_{j'}$ for all $j,j'\in\llbracket1,m\rrbracket$ such that $j<j'$. By convention, we will say that $i_{0}\preceq i$ and $i\prec i_{m+1}$ whenever $i\in\Supp(a)$. We define the following $m+1$ subsets of $\Supp(a)$ for any $l\in\llbracket0,m\rrbracket$:
\begin{equation*}
I_{l}\coloneqq\{i\in\Supp(a)\mid i_{l}\preceq i\prec i_{l+1}\}\text{.}
\end{equation*}

Let $a$ be a weight function. We set:
\begin{gather*}
\val_{p}(a)\coloneqq\min\{\val_{p}(a_{i})\mid i\in\llbracket1,n\rrbracket\}\text{,}\\
u(a)\coloneqq\max\{0,-\val_{p}(a)\}\text{.}
\end{gather*}

If $a$ is not the zero function, we put:
\begin{equation*}
g(a)\coloneqq F^{u(a)+\val_{p}(a)}\left(d\left(V^{u(a)}\left(\left[\underline{X}^{p^{-\val_{p}(a)}a|_{I_{0}}}\right]\right)\right)\right)\text{.}
\end{equation*}

Furthermore, if $I$ is a partition of $a$, and for any $\eta\in W(k)$, we set:
\begin{equation}\label{eelement}
e(\eta,a,I)\coloneqq\left\{\begin{array}{ll}d\left(V^{u(a)}\left(\eta\left[\underline{X}^{p^{u(a)}a|_{I_{0}}}\right]\right)\right)\times\prod_{l=2}^{\#I}g(a|_{I_{l}})&\text{if }I_{0}=\emptyset\text{ and }u(a)\neq0\text{,}\\V^{u(a)}\left(\eta\left[\underline{X}^{p^{u(a)}a|_{I_{0}}}\right]\right)\times\prod_{l=1}^{\#I}g(a|_{I_{l}})&\text{otherwise.}\end{array}\right.
\end{equation}

Notice that, if one ignores $\eta$, the element defined above is a product of $\#I$ factors whenever $I_{0}=\emptyset$, and of $\#I+1$ factors otherwise. We are going to use this property later on, when we will define the pseudovaluations on the de Rham-Witt complex of a polynomial ring.

We recall the action of $d$, $V$ and $F$ on these elements.

\begin{prop}\label{dactionone}
For any $(a,I)\in\mathcal{P}$ and any $\eta\in W(k)$ we have:
\begin{equation*}
d(e(\eta,a,I))=\left\{\begin{array}{ll}0&\text{if }I_{0}=\emptyset\text{,}\\e(\eta,a,I\cup\{\min(a)\})&\text{if }I_{0}\neq\emptyset\text{ and }\val_{p}(a)\leqslant0\text{,}\\p^{\val_{p}(a)}e(\eta,a,I\cup\{\min(a)\})&\text{if }I_{0}\neq\emptyset\text{ and }\val_{p}(a)>0\text{.}\end{array}\right.
\end{equation*}
\end{prop}

\begin{proof}
See \cite[Proposition 2.6]{derhamwittcohomologyforaproperandsmoothmorphism}.
\end{proof}

\begin{prop}\label{factionone}
For any $(a,I)\in\mathcal{P}$ and any $\eta\in W(k)$ we have:
\begin{equation*}
F(e(\eta,a,I))=\left\{\begin{array}{ll}e\left(\eta,pa,I\right)&\text{if }\val_{p}(a)<0\text{ and }I_{0}=\emptyset\text{,}\\e\left(p\eta,pa,I\right)&\text{if }\val_{p}(a)<0\text{ and }I_{0}\neq\emptyset\text{,}\\e\left(F(\eta),pa,I\right)&\text{if }\val_{p}(a)\geqslant0\text{.}\end{array}\right.
\end{equation*}
\end{prop}

\begin{proof}
See \cite[Proposition 2.5]{derhamwittcohomologyforaproperandsmoothmorphism}.
\end{proof}

\begin{prop}\label{vactionone}
For any $(a,I)\in\mathcal{P}$ and any $\eta\in W(k)$ we have:
\begin{equation*}
V(e(\eta,a,I))=\left\{\begin{array}{ll}e\left(V(\eta),\frac{a}{p},I\right)&\text{if }\val_{p}(a)>0\text{,}\\e\left(p\eta,\frac{a}{p},I\right)&\text{if }\val_{p}(a)\leqslant0\text{ and }I_{0}=\emptyset\text{,}\\e\left(\eta,\frac{a}{p},I\right)&\text{if }\val_{p}(a)\leqslant0\text{ and }I_{0}\neq\emptyset\text{.}\end{array}\right.
\end{equation*}
\end{prop}

\begin{proof}
See \cite[Proposition 2.5]{derhamwittcohomologyforaproperandsmoothmorphism}.
\end{proof}

The following theorem is essential to the definition of the overconvergent de Rham-Witt complex, and to its decomposition as a $W(k)$-module in the case of a polynomial algebra.

\begin{thrm}\label{structuretheorem}
For any differential $x\in W\Omega_{k[\underline{X}]/k}$ there exists an unique function $\eta\colon\begin{array}{rl}\mathcal{P}\to&W(k)\\(a,I)\mapsto&\eta_{a,I}\end{array}$ such that:
\begin{equation*}
x=\sum_{(a,I)\in\mathcal{P}}e(\eta_{a,I},a,I)\text{.}
\end{equation*}

Moreover, a series of that form converges in $W\Omega_{k[\underline{X}]/k}$ if and only if for any $m\in\mathbb{N}$ we have $V^{u(a)}(\eta_{a,I})\in V^{m}(W(k))$ except for a finite number of $(a,I)\in\mathcal{P}$.
\end{thrm}

\begin{proof}
See \cite[Theorem 2.8]{derhamwittcohomologyforaproperandsmoothmorphism}.
\end{proof}

\section{Computations}

Let $k$ be a commutative $\mathbb{Z}_{(p)}$-algebra. Let $n\in\mathbb{N}$, and let $k[\underline{X}]\coloneqq k[X_{1},\ldots,X_{n}]$. For any $(a,I)\in\mathcal{P}$ with $a$ taking values in $\mathbb{N}$, we will write $\underline{X}^{k}\coloneqq\prod_{j\in J}{X_{j}}^{a_{j}}$ and:
\begin{equation*}
h(a,I)\coloneqq\prod_{i\in\Supp(a)\smallsetminus I}\left[X_{i}\right]^{k_{i}}\prod_{j\in I}g(a|_{\{j\}})\text{.}
\end{equation*}

The goal of this section is to make explicit computations of the product of two elements of the form \eqref{eelement}. We will use the elements $h$ defined above in order to achieve this, as they appear more convenient for calculations.

\begin{lemm}\label{teichmullerlemma}
Let $R$ be a commutative $k$-algebra. Let $x\in R$ and let $m,m'\in\mathbb{N}$ such that $m+m'\neq0$. Put $a\coloneqq\val_{p}(m+m')$ and $b\coloneqq p^{-a}(m+m')$. Then:
\begin{equation*}
[x]^{m}d\left([x]^{m'}\right)=\frac{m'}{b}F^{a}\left(d\left([x]^{b}\right)\right)\text{.}
\end{equation*}
\end{lemm}

\begin{proof}
One has $\left(m+m'\right)[x]^{m}d\left([x]^{m'}\right)=m'd\left([x]^{m+m'}\right)$, therefore using \eqref{dfpfd} we get:
\begin{equation*}
d\left([x]^{m+m'}\right)=d\left(F^{a}\left(\left[x\right]^{b}\right)\right)=p^{a}F^{a}\left(d\left(\left[x\right]^{b}\right)\right)\text{.}
\end{equation*}

In the case where $k=\mathbb{Z}_{(p)}$, $R=\mathbb{Z}_{(p)}[X]$ and $x=X$, this means we have shown that:
\begin{equation*}
p^{a}b[x]^{m}d\left([x]^{m'}\right)=p^{a}m'F^{a}\left(d\left([x]^{b}\right)\right)\text{.}
\end{equation*}

We can conclude in that case using theorem \ref{structuretheorem}, and the fact that $W(\mathbb{Z}_{(p)})$ is a $\mathbb{Z}_{(p)}$-algebra with no $p$-torsion. For the general case, using the canonical commutative diagram
\begin{equation*}
\begin{tikzcd}
\mathbb{Z}_{(p)}[X]\arrow{r}&R\\
\mathbb{Z}_{(p)}\arrow{u}\arrow{r}&k\arrow{u}
\end{tikzcd}
\end{equation*}
where the upper arrow sends $X$ to $x$, we conclude using the morphism of $W(\mathbb{Z}_{(p)})$-dgas $W\Omega_{\mathbb{Z}_{(p)}[X]/\mathbb{Z}_{(p)}}\to W\Omega_{R/k}$ obtained by functoriality.
\end{proof}

The next proposition gives a simple formula for products of values of $h$. The goal of the following lemmas will be to use it in order to get a formula for products of elements of the form \eqref{eelement}.

\begin{prop}\label{hhproduct}
Let $(a,I),(b,J)\in\mathcal{P}$ such that $a$ and $b$ take values in $\mathbb{N}$. We have:
\begin{equation*}
\exists m\in\mathbb{Z}_{(p)},\ h(a,I)h(b,J)=\left\{\begin{array}{ll}m\,h(a+b,I\cup J)&\text{if }I\cap J=\emptyset\\0&\text{otherwise.}\end{array}\right.
\end{equation*}
\end{prop}

\begin{proof}
First, we have by definition:
\begin{equation*}
h(a,I)h(b,J)=\prod_{i\in\Supp(a)\smallsetminus I}\left[X_{i}\right]^{a_{i}}\prod_{i'\in I}g\left(a|_{\{i'\}}\right)\prod_{j\in\Supp(b)\smallsetminus J}\left[X_{j}\right]^{b_{j}}\prod_{j'\in J}g\left(b|_{\{j'\}}\right)\text{.}
\end{equation*}

Since $W\Omega_{k[\underline{X}]/k}$ is alternating, this product is zero whenever $I\cap J\neq\emptyset$. Otherwise, we get:
\begin{multline*}
h(a,I)h(b,J)\\=\prod_{i\in\Supp(a+b)\smallsetminus(I\cup J)}\left[X_{i}\right]^{a_{i}+b_{i}}\prod_{i'\in I}\left[X_{i'}\right]^{b_{i'}}g\left(a|_{\{i'\}}\right)\prod_{j'\in J}\left[X_{j'}\right]^{a_{j'}}g\left(b|_{\{j'\}}\right)\text{.}
\end{multline*}

Moreover, for any $i'\in I$ we check that:
\begin{equation*}
\begin{split}
\left[X_{i'}\right]^{b_{i'}}g\left(a|_{\{i'\}}\right)&\overset{\eqref{fdttdt}}{=}\left[X_{i'}\right]^{\left(1-p^{-\val_{p}\left(a|_{\{i'\}}\right)}\right)a_{i'}+b_{i'}}d\left(\left[X_{i}\right]^{p^{-\val_{p}\left(a|_{\{i'\}}\right)}a_{i'}}\right)\\
&\overset{\hphantom{\eqref{fdttdt}}}{\overset{\ref{teichmullerlemma}}{=}}\frac{p^{-\val_{p}\left(a|_{\{i'\}}\right)}a_{i'}}{p^{-\val_{p}\left(a_{i'}+b_{i'}\right)}\left(a_{i'}+b_{i'}\right)}g\left((a+b)|_{\{i'\}}\right)\text{.}
\end{split}
\end{equation*}

Using the same argument, for any $j'\in J$ one successfully gets:
\begin{equation*}
\left[X_{j'}\right]^{a_{j'}}g\left(b|_{\{j'\}}\right)=\frac{p^{-\val_{p}\left(b|_{\{j'\}}\right)}b_{j'}}{p^{-\val_{p}\left(a_{j'}+b_{j'}\right)}\left(a_{j'}+b_{j'}\right)}g\left((a+b)|_{\{j'\}}\right)\text{.}
\end{equation*}
\end{proof}

\begin{lemm}\label{gashs}
Let $(a,I)\in\mathcal{P}$ such that $a$ takes values in $\mathbb{N}$. We have:
\begin{equation*}
g(a)=\sum_{j\in\Supp(a)}p^{\val_{p}(a_{j})-\val_{p}(a)}h(a,\{j\})\text{.}
\end{equation*}
\end{lemm}

\begin{proof}
Write $S\coloneqq\Supp(a)$ for simplicity. We compute:
\begin{multline*}
F^{\val_{p}(a)}\left(d\left(\left[\underline{X}^{p^{-\val_{p}(a)}a}\right]\right)\right)\\\begin{aligned}
&\overset{\eqref{fdttdt}}{=}\left[\underline{X}^{\left(1-p^{-\val_{p}(a)}\right)a}\right]d\left(\left[\underline{X}^{p^{-\val_{p}(a)}a}\right]\right)\\
&\overset{\eqref{dproducts}}{=}\left[\underline{X}^{\left(1-p^{-\val_{p}(a)}\right)a}\right]\sum_{j\in S}\left(\prod_{j'\in S\smallsetminus\{j\}}\left[{X_{j'}}^{p^{-\val_{p}(a)}a_{j'}}\right]\right)d\left(\left[{X_{j}}^{p^{-\val_{p}(a)}a_{j}}\right]\right)\\
&\overset{\eqref{fdttdt}}{=}\sum_{j\in S}\left(\prod_{j'\in S\smallsetminus\{j\}}\left[{X_{j'}}^{a_{j'}}\right]\right)F^{\val_{p}(a)}\left(d\left(\left[{X_{j}}^{p^{-\val_{p}(a)}a_{j}}\right]\right)\right)\\
&\overset{\eqref{dfpfd}}{=}\sum_{j\in S}\left(\prod_{j'\in S\smallsetminus\{j\}}\left[{X_{j'}}^{a_{j'}}\right]\right)p^{\val_{p}(a_{j})-\val_{p}(a)}F^{\val_{p}(a_{j})}\left(d\left(\left[{X_{j}}^{p^{-\val_{p}(a_{j})}a_{j}}\right]\right)\right)\\
&\overset{\hphantom{\eqref{dfpfd}}}{=}\sum_{j\in S}p^{\val_{p}(a_{j})-\val_{p}(a)}h(a,\{j\})\text{.}\end{aligned}
\end{multline*}

This ends the proof because by definition $g(a)=F^{\val_{p}(a)}\left(d\left(\left[\underline{X}^{p^{-\val_{p}(a)}a}\right]\right)\right)$. 
\end{proof}

The next lemma will be used so we can write any value of the function $h$ defined above as a linear combination of elements of the form \eqref{eelement}. The previous lemma can be seen as a kind of reciprocal.

\begin{lemm}\label{hases}
Let $(a,I)\in\mathcal{P}$ such that $a$ takes values in $\mathbb{N}$. Denote by $P$ the set of partitions of $\Supp(a)$ of size $\#I$. Then there exists a function $s\colon P\to\mathbb{N}\subset W(k)$ such that:
\begin{equation*}
h(a,I)=\sum_{J\in P}e(s(J),a,J)\text{.}
\end{equation*}
\end{lemm}

\begin{proof}
Put $m\coloneqq\#I$. If $m=0$, then obviously $h(a,I)=e(1,a,I)$. So suppose that $m\neq0$. Write $\{i_{l}\}_{l\in\llbracket1,m\rrbracket}\coloneqq I$, with $i_{j}\prec i_{j'}$ for any pair $j<j'$ in $\llbracket1,m\rrbracket$, and for all $j\in\Supp(a)$ put $v_{j}\coloneqq\val_{p}(a_{j})$ and $b_{j}=p^{-v_{j}}a_{j}$. By definition:
\begin{equation*}
h(a,I)=\prod_{i\in\Supp(a)\smallsetminus I}\left[X_{i}\right]^{a_{i}}\prod_{j\in I}F^{v_{j}}\left(d\left(\left[X_{j}\right]^{b_{j}}\right)\right)\text{.}
\end{equation*}

So we can write:
\begin{equation*}
h(a,I)=h\left(a|_{\Supp(a)\smallsetminus I_{m}},I\smallsetminus\{i_{m}\}\right)\times\left(\prod_{i\in I_{m}\smallsetminus\{i_{m}\}}\left[X_{i}\right]^{a_{i}}\right)\times F^{v_{i_{m}}}\left(d\left(\left[X_{i_{m}}\right]^{b_{i_{m}}}\right)\right)\text{.}
\end{equation*}

Moreover we can compute:
\begin{multline*}
\left(\prod_{i\in I_{m}\smallsetminus\{i_{m}\}}\right.\left.\vphantom{\prod_{i\in I_{m}\smallsetminus\{i_{m}\}}}\left[X_{i}\right]^{a_{i}}\right)\times F^{v_{i_{m}}}\left(d\left(\left[X_{i_{m}}\right]^{b_{i_{m}}}\right)\right)\\\begin{aligned}
&\overset{\eqref{frobwittvect}}{=}F^{v_{i_{m}}}\left(d\left(\left[X_{i_{m}}\right]^{b_{i_{m}}}\right)\prod_{i\in I_{m}\smallsetminus\{i_{m}\}}\left[X_{i}\right]^{p^{-v_{i_{m}}}a_{i}}\right)\\
&\overset{\eqref{dbasicprop}}{=}F^{v_{i_{m}}}\left(d\left(\prod_{i\in I_{m}}\left[X_{i}\right]^{p^{-v_{i_{m}}}a_{i}}\right)-\left[X_{i_{m}}\right]^{b_{i_{m}}}d\left(\prod_{i\in I_{m}\smallsetminus\{i_{m}\}}\left[X_{i}\right]^{p^{-v_{i_{m}}}a_{i}}\right)\right)\\
&\overset{\eqref{dfpfd}}{=}g\left(a|_{I_{m}}\right)-F^{v_{i_{m}}}\left(\left[X_{i_{m}}\right]^{b_{i_{m}}}\right)\times p^{\val_{p}\left(a|_{I_{m}\smallsetminus\{i_{m}\}}\right)-v_{i_{m}}}g\left(a|_{I_{m}\smallsetminus\{i_{m}\}}\right)\\
&\overset{\eqref{frobwittvect}}{=}g\left(a|_{I_{m}}\right)-p^{\val_{p}\left(a|_{I_{m}\smallsetminus\{i_{m}\}}\right)-v_{i_{m}}}\left[X_{i_{m}}\right]^{a_{i_{m}}}g\left(a|_{I_{m}\smallsetminus\{i_{m}\}}\right)\text{.}\end{aligned}
\end{multline*}

So we get:
\begin{multline*}
h(a,I)=h\left(a|_{\Supp(a)\smallsetminus I_{m}},I\smallsetminus\{i_{m}\}\right)\times g\left(a|_{I_{m}}\right)\\-p^{\val_{p}\left(a|_{I_{m}\smallsetminus\{i_{m}\}}\right)-v_{i_{m}}}h\left(a|_{\{i_{m}\}\cup\Supp(a)\smallsetminus I_{m}},I\smallsetminus\{i_{m}\}\right)\times g\left(a|_{I_{m}\smallsetminus\{i_{m}\}}\right)\text{.}
\end{multline*}

We can then deduce the lemma by induction on $m=\#I$. Indeed, if we suppose that $h\left(a|_{\Supp(a)\smallsetminus I_{m}},I\smallsetminus\{i_{m}\}\right)$ can be written as a linear combination of elements of the form $e(1,a|_{\Supp(a)\smallsetminus I_{m}},J')$, where $J'$ is a partition of $\Supp(a)\smallsetminus I_{m}$ of size $m-1$, then since $e(1,a|_{\Supp(a)\smallsetminus I_{m}},J')\times g\left(a|_{I_{m}}\right)=e(1,a,J'\cup\{i_{m}\})$ the lemma is proven for the minuend of the above subtraction, and one can conclude for the subtrahend by using the same reasoning.
\end{proof}

\begin{lemm}\label{eeproductint}
Let $(a,I),(b,J)\in\mathcal{P}$ such that $a$ and $b$ take values in $\mathbb{N}$. Let $\eta,\eta'\in W(k)$. Denote by $P$ the set of partitions of $\Supp(a+b)$ of size $\#I+\#J$, then there exists a function $s\colon P\to\mathbb{Z}_{(p)}$ such that:
\begin{equation*}
e(\eta,a,I)e(\eta',b,J)=\sum_{L\in P}e(s(L)\eta\eta',a+b,L)\text{.}
\end{equation*}
\end{lemm}

\begin{proof}
By definition we have $e(\eta,a,I)=\eta\left[\underline{X}^{a|_{I_{0}}}\right]\prod_{i=1}^{\#I}g(a|_{I_{i}})$. There is also a similar equation defining $e(\eta',b,J)$. We then deduce from lemma \ref{gashs} that $e(\eta,a,I)$ is a linear combination of products of elements of the form $\eta h(a|_{I_{0}},\emptyset)\prod_{i=1}^{\#I}h(a|_{I_{i}},\{j_{i}\})$, where all $j_{i}\in I_{i}$ for any $i\in\llbracket1,\#I\rrbracket$. We can conclude by using proposition \ref{hhproduct} and lemma \ref{hases}.
\end{proof}

\begin{lemm}\label{protolemma}
Let $(a,I),(b,J)\in\mathcal{P}$ such that $u(a)\geqslant u(b)$ and $I_{0}\neq\emptyset$. Denote by $P$ the set of partitions of size $\#I+\#J$ of $\Supp(a+b)$, and put:
\begin{equation*}
v\coloneqq\left\{\begin{array}{ll}u(b)&\text{if }J_{0}\neq\emptyset\text{,}\\0&\text{otherwise.}\end{array}\right.
\end{equation*}

Then for any $\eta,\eta'\in W(k)$ there exists a function $s\colon P\to\mathbb{Z}_{(p)}$ such that:
\begin{gather*}
\forall L\in P,\ \left\{\begin{array}{ll}p^{v+u(a+b)}\mid s(L)&\text{if }L_{0}=\emptyset\text{,}\\p^{v}\mid s(L')&\text{otherwise,}\end{array}\right.\\
e(\eta,a,I)e(\eta',b,J)=\sum_{L\in P}e\left(s(L)V^{u(a)-u(a+b)}\left(\eta F^{u(a)-u(b)}\left(\eta'\right)\right),a+b,L\right)\text{.}
\end{gather*}
\end{lemm}

\begin{proof}
Put $\tilde{I}\coloneqq\bigcup_{i\in\llbracket1,\#I\rrbracket}I_{i}$. We can compute:
\begin{multline*}
e(\eta,a,I)e(\eta',b,J)\\\begin{aligned}
&\overset{\eqref{vxfyvxy}}{=}V^{u(a)}\left(\eta\left[\underline{X}^{p^{u(a)}(a|_{I_{0}})}\right]F^{u(a)}\left(e(1,a|_{\tilde{I}},I)e(\eta',b,J)\right)\right)\\
&\overset{\hphantom{\eqref{vxfyvxy}}}{\overset{\ref{factionone}}{=}}V^{u(a)}\left(\eta\left[\underline{X}^{p^{u(a)}(a|_{I_{0}})}\right]e\left(1,p^{u(a)}a|_{\tilde{I}},I\right)e\left(p^{v}F^{u(a)-u(b)}(\eta'),p^{u(a)}b,J\right)\right)\\
&\overset{\hphantom{\eqref{vxfyvxy}}}{=}V^{u(a)}\left(e(\eta,p^{u(a)}a,I)e\left(p^{v}F^{u(a)-u(b)}(\eta'),p^{u(a)}b,J\right)\right)\text{.}\end{aligned}
\end{multline*}

According to lemma \ref{eeproductint}, there is a function $s'\colon P\to\mathbb{Z}_{(p)}$ such that:
\begin{equation*}
e(a,I)e(b,J)=V^{u(a)}\left(\sum_{L\in P}e\left(p^{v}s'(L)\eta F^{u(a)-u(b)}(\eta'),p^{u(a)}(a+b),L\right)\right)\text{.}
\end{equation*}

We can conclude by using proposition \ref{vactionone}, and the fact that the \emph{Verschiebung} endomorphism is additive.
\end{proof}

In the last two propositions of this section we are interested in the case where $k$ has characteristic $p$. The results become clearer in this case because we have $p=V(F(1))$.

\begin{lemm}\label{technicallem}
Suppose $k$ has characteristic $p$. Let $(a,I),(b,J)\in\mathcal{P}$ such that $u(a)\geqslant u(b)$ and $I_{0}\neq\emptyset$. Denote by $P$ the set of partitions of size $\#I+\#J$ of $\Supp(a+b)$, and put:
\begin{equation*}
v\coloneqq\left\{\begin{array}{ll}u(b)&\text{if }J_{0}\neq\emptyset\text{,}\\0&\text{otherwise.}\end{array}\right.
\end{equation*}

Let $\alpha,\beta\in\mathbb{N}$. Then for any $\eta\in V^{\alpha}(W(k))$ and any $\eta'\in V^{\beta}(W(k))$ there exists a function $s\colon P\to W(k)$ such that:
\begin{gather*}
\forall L\in P,\ \left\{\begin{array}{ll}s(L)\in V^{\alpha+\beta+v+u(a)}(W(k))&\text{if }L_{0}=\emptyset\text{,}\\s(L)\in V^{\alpha+\beta+v+u(a)-u(a+b)}(W(k))&\text{otherwise,}\end{array}\right.\\
e(\eta,a,I)e(\eta',b,J)=\sum_{L\in P}e\left(s(L),a+b,L\right)\text{.}
\end{gather*}
\end{lemm}

\begin{proof}
This is a special case of lemma \ref{protolemma} when $k$ has characteristic $p$, because in that case we have $px=F(V(x))=V(F(x))$ for any $x\in W(k)$, but also $\eta\eta'\in V^{\alpha+\beta}(W(k))$ \cite[Proposition 5. p. IX.15]{bourbakicommutativehuit}.
\end{proof}

\begin{prop}\label{mainprop}
Suppose $k$ has characteristic $p$. Let $(a,I),(b,J)\in\mathcal{P}$ with $I_{0}\neq\emptyset$. Denote by $P$ the set of partitions of size $\#I+\#J$ of $\Supp\left(a+b\right)$. Let $\alpha,\beta\in\mathbb{N}$. Then for any $\eta\in V^{\alpha}(W(k))$ and any $\eta'\in V^{\beta}(W(k))$ there exists a function $s\colon P\to W(k)$ such that:
\begin{gather*}
\forall L\in P,\ \left\{\begin{array}{ll}s(L)\in V^{\alpha+\beta+\min\{u(a),u(b)\}}(W(k))&\text{if }L_{0}=\emptyset\text{,}\\s(L)\in V^{\alpha+\beta+\max\{u(a),u(b)\}-u(a+b)}(W(k))&\text{otherwise,}\end{array}\right.\\
e(\eta,a,I)e(\eta',b,J)=\sum_{L\in P}e(s(L),a+b,L)\text{.}
\end{gather*}
\end{prop}

\begin{proof}
This statement is a particular case of lemma \ref{technicallem}, except when $u(b)>u(a)$ and $J_{0}=\emptyset$. In that situation, if $J'\coloneqq J\smallsetminus\{\min(b)\}$ we deduce from proposition \ref{dactionone} that:
\begin{equation*}
\begin{split}
e(\eta,a,I)e(\eta',b,J)&=e(\eta,a,I)d(e(\eta,b,J'))\\
&=(-1)^{\#I}\left(d(e(\eta,a,I)e(\eta',b,J'))-d(e(\eta,a,I))e(\eta',b,J')\right)\text{.}
\end{split}
\end{equation*}

This enables us to conclude using lemma \ref{technicallem} again.
\end{proof}

\section{Pseudovaluations}

We shall now consider the case where $k$ is a commutative ring of characteristic $p$. Let $n\in\mathbb{N}$, and let $k[\underline{X}]\coloneqq k[X_{1},\ldots,X_{n}]$. Recall that theorem \ref{structuretheorem} says that any $x\in W\Omega_{k[\underline{X}]/k}$ can be uniquely written as $\sum_{(a,I)\in\mathcal{P}}e(\eta_{a,I},a,I)$, where all $\eta_{a,I}\in W(k)$. This allows us to define specific $W(k)$-submodules of the de Rham-Witt complex.

Any $x=\sum_{(a,I)\in\mathcal{P}}e(\eta_{a,I},a,I)\in W\Omega_{k[\underline{X}]/k}$ is said to be \textbf{integral} if $\eta_{a,I}=0$ anytime $u(a)\neq0$. We denote by $W\Omega^{\mathrm{int}}_{k[\underline{X}]/k}$ the subset of all integral elements of the de Rham-Witt complex.

The element $x$ is said to be \textbf{fractional} if $\eta_{a,I}=0$ whenever $u(a)=0$. We denote by $W\Omega^{\mathrm{frac}}_{k[\underline{X}]/k}$ the subset of all fractional elements of the de Rham-Witt complex.

The element $x$ is said to be \textbf{pure fractional} if $\eta_{a,I}=0$ anytime $u(a)=0$ or $I_{0}=\emptyset$. We denote by $W\Omega^{\mathrm{frp}}_{k[\underline{X}]/k}$ the subset of all pure fractional elements of the de Rham-Witt complex.

Notice that we have the following decomposition as $W(k)$-modules:
\begin{equation}\label{decomposition}
W\Omega_{k[\underline{X}]/k}\cong W\Omega_{k[\underline{X}]/k}^{\mathrm{int}}\oplus W\Omega_{k[\underline{X}]/k}^{\mathrm{frp}}\oplus d\left(W\Omega_{k[\underline{X}]/k}^{\mathrm{frp}}\right)\text{.}
\end{equation}

In all this chapter, for any $x\in W\Omega_{k[\underline{X}]/k}$, we will write $x|_{\mathrm{int}}$, $x|_{\mathrm{frac}}$, $x|_{\mathrm{frp}}$ and $x|_{\mathrm{d(frp)}}$ the obvious projections for this decomposition.

We will also denote by $\val_{V}$ the $V$-adic pseudovaluation on $W(k)$. Davis, Langer and Zink have defined the following functions for any $\varepsilon>0$:
\begin{equation*}
\gamma_{\varepsilon}\colon\begin{array}{rl}W\Omega_{k[\underline{X}]/k}\to&\mathbb{R}\cup\{+\infty,-\infty\}\\\sum_{(a,I)\in\mathcal{P}}e(\eta_{a,I},a,I)\mapsto&\inf_{(a,I)\in\mathcal{P}}\{\val_{V}(\eta_{a,I})+u(a)-\varepsilon\lvert a\rvert\}\text{.}\end{array}
\end{equation*}

The overconvergent de Rham-Witt complex of $k[\underline{X}]$ is the set of all $x\in W\Omega_{k[\underline{X}]/k}$ such that there exists $\varepsilon>0$ with $\gamma_{\varepsilon}(x)\neq-\infty$.

One of the main obstacles to studying the overconvergence of recursive sequences containing products of de Rham-Witt differentials is that these functions are not pseudovaluations. We will first study two counterexamples to the product rule in the case where $k[\underline{X}]\cong k[X,Y]$ as $k$-algebras. That is, we will find $x,y\in W\Omega_{k[\underline{X}]/k}$ such that for all $\varepsilon>0$ we have $\gamma_{\varepsilon}(x)\neq-\infty$, $\gamma_{\varepsilon}(y)\neq-\infty$ and $\gamma_{\varepsilon}(xy)<\gamma_{\varepsilon}(x)+\gamma_{\varepsilon}(y)$.

For any $m\in\mathbb{N}$, notice that:
\begin{gather*}
V^{m}\left(\left[X^{p^{m}-1}\right]\right)d\left(V^{m}\left([X]\right)\right)=p^{m}d([X])\text{,}\\
\gamma_{\varepsilon}\left(V^{m}\left(\left[X^{p^{m}-1}\right]\right)\right)=m-\frac{\varepsilon(p^{m}-1)}{p^{m}}\text{,}\\
\gamma_{\varepsilon}\left(d\left(V^{m}\left([X]\right)\right)\right)=m-\frac{\varepsilon}{p^{m}}\text{,}\\
\gamma_{\varepsilon}(p^{m}d([X]))=m-\varepsilon<2m-\varepsilon\text{.}
\end{gather*}

This first counterexample happens when taking the product of two fractional elements. The phenomena happening here is that the power of the denominator of the weight functions (that we denoted $a\mapsto u(a)$) can get smaller when taking products of differentials. Indeed, we have seen in propositions such as \ref{mainprop} that multiplying basic elements translates as an addition for weight functions. However, we notice in this example that the $V$-pseudovaluation we have to calculate gets bigger, it is just not big enough so it compensates the decrease of $u$. In this example, it seems to be enough to multiply the $V$-pseudovaluation by $2$. It still is not sufficient in general, as seen in this second counterexample:
\begin{gather*}
\gamma_{\varepsilon}\left(V^{m}\left(\left[X^{p^{m}-1}\right]\right)\right)=m-\frac{\varepsilon(p^{m}-1)}{p^{m}}\text{,}\\
\gamma_{\varepsilon}\left(d\left(V^{m}\left([Y]\right)\right)\right)=m-\frac{\varepsilon}{p^{m}}\text{,}\\
\gamma_{\varepsilon}(V^{m}\left(\left[X^{p^{m}-1}\right]\right)d\left(V^{m}\left([Y]\right)\right))=m-\varepsilon<2m-\varepsilon\text{.}
\end{gather*}

In this situation, the value of the function $u$ always stays the same. The reason is that $V^{m}\left(\left[X^{p^{m}-1}\right]\right)d\left(V^{m}\left([Y]\right)\right)$ is already an element of the form \eqref{eelement}, so no simplifications are needed as in the first counterexample. However, in all cases the function $u$ is only counted once in $\gamma_{\varepsilon}$, even though it should be counted twice in the product. In order for the product formula to work in general, we need to multiply $u$ by the number of factors in the definition of \eqref{eelement}. As this number is smaller than $n$, as remarked after the first counterexample we also have to multiply the $V$-pseudovaluation by $2n$. This leads us to the definition below.

For any $\varepsilon>0$ put:
\begin{equation*}
\zeta_{\varepsilon}\colon\begin{array}{rl}W\Omega_{k[\underline{X}]/k}\to&\mathbb{R}\cup\{+\infty,-\infty\}\\x\mapsto&\left\{\begin{array}{ll}\inf_{(a,I)\in\mathcal{P}}\{2n\val_{V}(\eta_{a,I})+\#Iu(a)-\varepsilon\lvert a\rvert\}&\text{if }I_{0}=\emptyset\text{,}\\\inf_{(a,I)\in\mathcal{P}}\{2n\val_{V}(\eta_{a,I})+(\#I+1)u(a)-\varepsilon\lvert a\rvert\}&\text{if }I_{0}\neq\emptyset\text{,}\end{array}\right.\end{array}
\end{equation*}
where we wrote $x=\sum_{(a,I)\in\mathcal{P}}e(\eta_{a,I},a,I)$.

We are going to prove that these functions are pseudovaluations. Before we check the product formula, we first give a few basic properties. It is for instance immediate that:
\begin{equation}\label{zetaaddition}
\forall x,y\in W\Omega_{k[\underline{X}]/k},\ \zeta_{\varepsilon}(x+y)\geqslant\min\left\{\zeta_{\varepsilon}(x),\zeta_{\varepsilon}(y)\right\}\text{.}
\end{equation}

Also, a consequence of proposition \ref{dactionone} is that:
\begin{equation}\label{detzetavareps}
\forall x\in W\Omega_{k[\underline{X}]/k},\ \zeta_{\varepsilon}(d(x))\geqslant\zeta_{\varepsilon}(x)\text{.}
\end{equation}

The following proposition tells us that we recover the definition of the overconvergent de Rham-Witt complex with these functions:

\begin{prop}
Let $x\in W\Omega_{k[\underline{X}]/k}$. There exists $\varepsilon>0$ such that $\gamma_{\varepsilon}(x)\neq-\infty$ if and only if $\zeta_{\varepsilon'}(x)\neq-\infty$ for some $\varepsilon'>0$.
\end{prop}

\begin{proof}
Notice that whenever $n\neq0$ we have:
\begin{equation*}
\forall x\in W\Omega_{k[\underline{X}]/k},\ 2n\gamma_{\frac{\varepsilon}{2n}}(x)\geqslant\zeta_{\varepsilon}(x)\geqslant\gamma_{\varepsilon}(x)\text{.}
\end{equation*}

This ends the proof except when $n=0$. But when $n=0$ then $W\Omega_{k[\underline{X}]/k}\cong W(k)$ as $W(k)$-dgas so there is nothing to do.
\end{proof}

We are now going to prove the product formula. We are doing it by exhaustion using the decomposition \eqref{decomposition}. Even though most of the proofs below follow the same, simple strategy, it is still interesting to proceed that way to get a stronger formula in some cases.

\begin{prop}\label{intintzeta}
For any $\varepsilon>0$ and any $x,y\in W\Omega_{k[\underline{X}]/k}^{\mathrm{int}}$ we have:
\begin{equation*}
\left(\zeta_{\varepsilon}(x)\neq-\infty\wedge\zeta_{\varepsilon}(y)\neq-\infty\right)\implies\zeta_{\varepsilon}(xy)\geqslant\zeta_{\varepsilon}(x)+\zeta_{\varepsilon}(y)\text{.}
\end{equation*}
\end{prop}

\begin{proof}
By definition of $W\Omega_{k[\underline{X}]/k}^{\mathrm{int}}$, we know that for all $(a,I),(b,J)\in\mathcal{P}$ there exists $\eta_{a,I},\eta_{b,J}'\in W(k)$ such that:
\begin{gather*}
x=\sum_{\substack{(a,I)\in\mathcal{P}\\u(a)=0}}e(\eta_{a,I},a,I)\text{,}\\
y=\sum_{\substack{(b,J)\in\mathcal{P}\\u(b)=0}}e(\eta'_{b,J},b,J)\text{.}
\end{gather*}

For any $(a,I),(b,J)\in\mathcal{P}$ such that $u(a)=u(b)=0$, using lemma \ref{eeproductint} we get:
\begin{equation*}
\zeta_{\varepsilon}\left(e(\eta_{a,I},a,I)e(\eta'_{b,J},b,J)\right)\geqslant2n\val_{V}(\eta_{a,I}\eta_{b,J}')+(\#I+\#J)u(a+b)-\varepsilon\lvert a+b\rvert\text{.}
\end{equation*}

Since $u(a+b)=0$ and $\val_{V}(\eta_{a,I}\eta_{b,J}')\geqslant\val_{V}(\eta_{a,I})+\val_{V}(\eta_{b,J}')$ because $k$ has characteristic $p$ \cite[Proposition 5. p. IX.15]{bourbakicommutativehuit}, we can conclude.
\end{proof}

\begin{prop}\label{frpintdfrpzeta}
For any $\varepsilon>0$, any $x\in W\Omega_{k[\underline{X}]/k}^{\mathrm{int}}$ and any $y\in W\Omega_{k[\underline{X}]/k}^{\mathrm{frp}}$ we have:
\begin{equation*}
\left(\zeta_{\varepsilon}(x)\neq-\infty\wedge\zeta_{\varepsilon}(y)\neq-\infty\right)\implies\zeta_{\varepsilon}((xy)|_{\mathrm{d(frp)}})\geqslant\zeta_{\varepsilon}(x)+\zeta_{\varepsilon}(y)+1\text{.}
\end{equation*}
\end{prop}

\begin{proof}
By definition of integral and pure fractional elements, we know that for all $(a,I),(b,J)\in\mathcal{P}$ there exists $\eta_{a,I}',\eta_{b,J}\in W(k)$ such that:
\begin{gather*}
x=\sum_{\substack{(b,J)\in\mathcal{P}\\u(b)=0}}e(\eta_{b,J},b,J)\text{,}\\
y=\sum_{\substack{(a,I)\in\mathcal{P}\\u(a)>0\\I_{0}\neq\emptyset}}e(\eta'_{a,I},a,I)\text{.}
\end{gather*}

Then for any $(a,I),(b,J)\in\mathcal{P}$ such that $u(a)>0$, $I_{0}\neq\emptyset$ and $u(b)=0$, lemma \ref{technicallem} gives us:
\begin{multline*}
\zeta_{\varepsilon}\left((e(\eta_{b,J},b,J)e(\eta'_{a,I},a,I))|_{\mathrm{d(frp)}}\right)\\\geqslant2n\left(\val_{V}(\eta_{b,J})+\val_{V}(\eta_{a,I}')+u(a)\right)+(\#I+\#J)u(a+b)-\varepsilon\lvert a+b\rvert\text{.}
\end{multline*}

But $u(a+b)=u(b)$, so $\zeta_{\varepsilon}\left((e(\eta_{b,J},b,J)e(\eta'_{a,I},a,I))|_{\mathrm{d(frp)}}\right)\geqslant \zeta_{\varepsilon}(x)+\zeta_{\varepsilon}(y)+1$, as needed.
\end{proof}

\begin{prop}\label{frpintzeta}
For any $\varepsilon>0$, any $j\in\mathbb{N}$, any $x\in W\Omega_{k[\underline{X}]/k}^{\mathrm{int},j}$ and any $y\in W\Omega_{k[\underline{X}]/k}^{\mathrm{frp}}$ we get:
\begin{equation*}
\left(\zeta_{\varepsilon}(x)\neq-\infty\wedge\zeta_{\varepsilon}(y)\neq-\infty\right)\implies\zeta_{\varepsilon}\left((xy)|_{\mathrm{frp}}\right)\geqslant\zeta_{\varepsilon}(x)+\zeta_{\varepsilon}(y)+j\text{.}
\end{equation*}
\end{prop}

\begin{proof}
By definition of integral and pure fractional elements, we know that for all $(a,I),(b,J)\in\mathcal{P}$ there exists $\eta_{a,I}',\eta_{b,J}\in W(k)$ such that:
\begin{gather*}
x=\sum_{\substack{(b,J)\in\mathcal{P}\\u(b)=0\\\#J=j}}e(\eta_{b,J},b,J)\text{,}\\
y=\sum_{\substack{(a,I)\in\mathcal{P}\\u(a)>0\\I_{0}\neq\emptyset}}e(\eta'_{a,I},a,I)\text{.}
\end{gather*}

Using lemma \ref{technicallem}, we know that for any $(a,I),(b,J)\in\mathcal{P}$ such that $u(a)>0$, $I_{0}\neq\emptyset$ and $u(b)=0$ we have:
\begin{multline*}
\zeta_{\varepsilon}\left((e(\eta_{b,J},b,J)e(\eta'_{a,I},a,I))|_{\mathrm{frp}}\right)\\\geqslant2n\left(\val_{V}(\eta_{b,J})+\val_{v}(\eta_{a,I}')\right)+(\#I+\#J+1)u(a+b)-\varepsilon\lvert a+b\rvert\text{.}
\end{multline*}

Furthermore, notice that $u(a+b)=u(a)>0$. Therefore, we obtain that $\zeta_{\varepsilon}\left((e(\eta_{b,J},b,J)e(\eta'_{a,I},a,I))|_{\mathrm{frp}}\right)\geqslant\zeta_{\varepsilon}(x)+\zeta_{\varepsilon}(y)+\#J$, which ends this proof.
\end{proof}

\begin{prop}
For any $\varepsilon>0$, any $x\in W\Omega_{k[\underline{X}]/k}^{\mathrm{int}}$ and any $y\in W\Omega_{k[\underline{X}]/k}$ we have:
\begin{equation*}
\left(\zeta_{\varepsilon}(x)\neq-\infty\wedge\zeta_{\varepsilon}(y)\neq-\infty\right)\implies\zeta_{\varepsilon}(xy)\geqslant\zeta_{\varepsilon}(x)+\zeta_{\varepsilon}(y)\text{.}
\end{equation*}
\end{prop}

\begin{proof}
Notice that it is sufficient to prove this in the case where $x,y\in W\Omega_{k[\underline{X}]/k}^{i}$, for some $i\in\mathbb{N}$. Recall that $y=y|_{\mathrm{int}}+y|_{\mathrm{frp}}+y|_{\mathrm{d(frp)}}$, and notice that $\zeta_{\varepsilon}(y)\geqslant\zeta_{\varepsilon}\left(y|_{\mathrm{int}}\right)$, $\zeta_{\varepsilon}(y)\geqslant\zeta_{\varepsilon}\left(y|_{\mathrm{frp}}\right)$ and $\zeta_{\varepsilon}(y)\geqslant\zeta_{\varepsilon}\left(y|_{\mathrm{d(frp)}}\right)$. Using lemma \ref{technicallem}, we check that $xy|_{\mathrm{frp}}\in W\Omega_{k[\underline{X}]/k}^{\mathrm{frac}}$. Therefore, using propositions \ref{intintzeta}, \ref{frpintdfrpzeta} and \ref{frpintzeta} as well as \eqref{zetaaddition}, we see that we only need to show that $\zeta_{\varepsilon}(xy|_{\mathrm{d(frp)}})\geqslant\zeta_{\varepsilon}(x)+\zeta_{\varepsilon}\left(y|_{\mathrm{d(frp)}}\right)$.

Let $y'\in W\Omega_{k[\underline{X}]/k}^{\mathrm{frp},i-1}$ be an element such that $d(y')=y|_{\mathrm{d(frp)}}$. Using proposition \ref{dactionone} we get $\zeta_{\varepsilon}(y')=\zeta_{\varepsilon}(y|_{\mathrm{frp}})$. But $xy|_{\mathrm{d(frp)}}=(-1)^{i}\left(d(xy')-d(x)y'\right)$, so one can conclude using \eqref{zetaaddition}, \eqref{detzetavareps} as well as propositions \ref{frpintdfrpzeta} and \ref{frpintzeta}.
\end{proof}

\begin{prop}
For any $\varepsilon>0$ and any $x,y\in W\Omega_{k[\underline{X}]/k}^{\mathrm{frp}}$ we get:
\begin{equation*}
\left(\zeta_{\varepsilon}(x)\neq-\infty\wedge\zeta_{\varepsilon}(y)\neq-\infty\right)\implies\zeta_{\varepsilon}((xy)|_{\mathrm{frp}})\geqslant\zeta_{\varepsilon}(x)+\zeta_{\varepsilon}(y)+1\text{.}
\end{equation*}
\end{prop}

\begin{proof}
It is enough to prove that for any $(a,I),(b,J)\in\mathcal{P}$ with $u(a)\neq0$, $I_{0}\neq\emptyset$, $u(b)\neq0$ and $J_{0}\neq\emptyset$, and any Witt vectors $\eta_{a,I},\eta_{b,J}\in W(k)$ we have the inequality $\zeta_{\varepsilon}\left(\left(e(\eta_{a,I},a,I)e(\eta_{b,J},b,J)\right)|_{\mathrm{frp}}\right)\geqslant\zeta_{\varepsilon}(e(\eta_{a,I},a,I))+\zeta_{\varepsilon}(e(\eta_{b,J},b,J))+1$. A consequence of lemma \ref{technicallem} is that:
\begin{multline*}
\zeta_{\varepsilon}\left(\left(e(\eta_{a,I},a,I)e(\eta_{b,J},b,J)\right)|_{\mathrm{frp}}\right)\\\geqslant2n\left(\val_{V}(\eta_{a,I})+\val_{V}(\eta_{b,J})+u(a)+u(b)-u(a+b)\right)\\+(\#I+\#J+1)u(a+b)-\varepsilon\lvert a+b\rvert\text{.}
\end{multline*}

Therefore, one can conclude if one has:
\begin{equation*}
2n\left(u(a)+u(b)-u(a+b)\right)+(\#I+\#J+1)u(a+b)\geqslant(\#I+1)u(a)+(\#J+1)u(b)+1\text{.}
\end{equation*}

Notice that $\#I+1\leqslant n$ and $\#J+1\leqslant n$ because we supposed that $I_{0}\neq\emptyset$ and $J_{0}\neq\emptyset$. Since $u(a+b)\leqslant\max\{u(a),u(b)\}$, we get:
\begin{multline*}
2n\left(u(a)+u(b)\right)+(\#I+\#J+1-2n)u(a+b)\\\geqslant2n\min\{u(a),u(b)\}+(\#I+\#J+1)\max\{u(a),u(b)\}\text{.}
\end{multline*}

This ends the proof whenever $n\neq0$. If $n=0$, there is nothing to show.
\end{proof}

\begin{prop}\label{frpdfrpfrpzeta}
Let $\varepsilon>0$, $x\in W\Omega_{k[\underline{X}]/k}^{\mathrm{frp}}$ and $y\in d\left(W\Omega_{k[\underline{X}]/k}^{\mathrm{frp}}\right)$. We have:
\begin{equation*}
\left(\zeta_{\varepsilon}(x)\neq-\infty\wedge\zeta_{\varepsilon}(y)\neq-\infty\right)\implies\zeta_{\varepsilon}((xy)|_{\mathrm{frp}})\geqslant\zeta_{\varepsilon}(x)+\zeta_{\varepsilon}(y)\text{.}
\end{equation*}
\end{prop}

\begin{proof}
We only have to demonstrate that for any $(a,I),(b,J)\in\mathcal{P}$ with $u(a)\neq0$, $I_{0}\neq\emptyset$, $u(b)\neq0$ and $J_{0}=\emptyset$, and any Witt vectors $\eta_{a,I},\eta_{b,J}\in W(k)$ we have the inequality $\zeta_{\varepsilon}\left(\left(e(\eta_{a,I},a,I)e(\eta_{b,J},b,J)\right)|_{\mathrm{frp}}\right)\geqslant\zeta_{\varepsilon}(e(\eta_{a,I},a,I))+\zeta_{\varepsilon}(e(\eta_{b,J},b,J))$. Using proposition \ref{mainprop}, one checks that:
\begin{multline*}
\zeta_{\varepsilon}\left(\left(e(\eta_{a,I},a,I)e(\eta_{b,J},b,J)\right)|_{\mathrm{frp}}\right)\\\geqslant2n\left(\val_{V}(\eta_{a,I})+\val_{V}(\eta_{b,J})+\max\{u(a),u(b)\}-u(a+b)\right)\\+(\#I+\#J+1)u(a+b)-\varepsilon\lvert a+b\rvert\text{.}
\end{multline*}

So we can conclude if we show:
\begin{equation*}
2n\left(\max\{u(a),u(b)\}-u(a+b)\right)+(\#I+\#J+1)u(a+b)\geqslant(\#I+1)u(a)+\#Ju(b)\text{.}
\end{equation*}

Notice that $\#I+1\leqslant n$ because we have supposed that $I_{0}\neq\emptyset$, and $\#J\leqslant n$. As $u(a+b)\leqslant\max\{u(a),u(b)\}$, we see that:
\begin{equation*}
2n\max\{u(a),u(b)\}+(\#I+\#J+1-2n)u(a+b)\geqslant(\#I+\#J+1)\max\{u(a),u(b)\}\text{.}
\end{equation*}
\end{proof}

\begin{prop}
For any $\varepsilon>0$ and any $x,y\in W\Omega_{k[\underline{X}]/k}^{\mathrm{frp}}$ we have:
\begin{equation*}
\left(\zeta_{\varepsilon}(x)\neq-\infty\wedge\zeta_{\varepsilon}(y)\neq-\infty\right)\implies\zeta_{\varepsilon}((xy)|_{\mathrm{d(frp)}})\geqslant\zeta_{\varepsilon}(x)+\zeta_{\varepsilon}(y)+3\text{.}
\end{equation*}
\end{prop}

\begin{proof}
We only have to check that for any $(a,I),(b,J)\in\mathcal{P}$ such that $u(a)\neq0$, $I_{0}\neq\emptyset$, $u(b)\neq0$ and $J_{0}\neq\emptyset$, and any Witt vectors $\eta_{a,I},\eta_{b,J}\in W(k)$ we have the inequality $\zeta_{\varepsilon}\left(\left(e(\eta_{a,I},a,I)e(\eta_{b,J},b,J)\right)|_{\mathrm{d(frp)}}\right)\geqslant\zeta_{\varepsilon}(e(\eta_{a,I},a,I))+\zeta_{\varepsilon}(e(\eta_{b,J},b,J))+3$. Due to lemma \ref{technicallem} we see that:
\begin{multline*}
\zeta_{\varepsilon}\left(\left(e(\eta_{a,I},a,I)e(\eta_{b,J},b,J)\right)|_{\mathrm{d(frp)}}\right)\\\geqslant2n\left(\val_{V}(\eta_{a,I})+\val_{V}(\eta_{b,J})+u(a)+u(b)\right)+(\#I+\#J)u(a+b)-\varepsilon\lvert a+b\rvert\text{.}
\end{multline*}

Therefore, the proof is complete if:
\begin{equation*}
2n\left(u(a)+u(b)\right)+(\#I+\#J)u(a+b)\geqslant(\#I+1)u(a)+(\#J+1)u(b)+3\text{.}
\end{equation*}

In the fractional part that we are studying, we necessarily have $u(a)>0$, $u(b)>0$ and $u(a+b)>0$. Moreover, $\#I+1\leqslant n$ and $\#J+1\leqslant n$ as we have supposed that $I_{0}\neq\emptyset$ and $J_{0}\neq\emptyset$. So we get:
\begin{equation*}
2n\left(u(a)+u(b)\right)+(\#I+\#J)u(a+b)\geqslant(\#I+1)u(a)+(\#J+1)u(b)+2+\#I+\#J\text{.}
\end{equation*}

If $\#I+\#J\neq0$, the proof is complete. Otherwise, it means that we are multiplying two Witt vectors. In particular, the projection on $d\left(W\Omega_{k[\underline{X}]/k}^{\mathrm{frp}}\right)$ is $0$, but $\zeta_{\varepsilon}(0)=+\infty$ so the proof becomes obvious.
\end{proof}

\begin{prop}
For any $\varepsilon>0$ and any $x,y\in W\Omega_{k[\underline{X}]/k}^{\mathrm{frp}}$ we have:
\begin{equation*}
\left(\zeta_{\varepsilon}(x)\neq-\infty\wedge\zeta_{\varepsilon}(y)\neq-\infty\right)\implies\zeta_{\varepsilon}((xy)|_{\mathrm{int}})\geqslant\zeta_{\varepsilon}(x)+\zeta_{\varepsilon}(y)+2\text{.}
\end{equation*}
\end{prop}

\begin{proof}
It is enough to prove that for any $(a,I),(b,J)\in\mathcal{P}$ such that $u(a)\neq0$, $I_{0}\neq\emptyset$, $u(b)\neq0$ and $J_{0}\neq\emptyset$, and for any Witt vectors $\eta_{a,I},\eta_{b,J}\in W(k)$ we have the inequality $\zeta_{\varepsilon}\left(\left(e(\eta_{a,I},a,I)e(\eta_{b,J},b,J)\right)|_{\mathrm{int}}\right)\geqslant\zeta_{\varepsilon}(e(\eta_{a,I},a,I))+\zeta_{\varepsilon}(e(\eta_{b,J},b,J))+2$. But using lemma \ref{technicallem} one gets:
\begin{equation*}
\zeta_{\varepsilon}\left(\left(e(\eta_{a,I},a,I)e(\eta_{b,J},b,J)\right)|_{\mathrm{int}}\right)\geqslant2n\left(\val_{V}(\eta_{a,I})+\val_{V}(\eta_{b,J})+u(a)+u(b)\right)-\varepsilon\lvert a+b\rvert
\end{equation*}
because in the integral part we always have $u(a+b)=0$, which ends the proof as $n\geqslant\#I+1$ and $n\geqslant\#J+1$ since we supposed that $I_{0}\neq\emptyset$ and $J_{0}\neq\emptyset$.
\end{proof}

\begin{prop}
For any $\varepsilon>0$ and any $x,y\in W\Omega_{k[\underline{X}]/k}^{\mathrm{frac}}$ we have:
\begin{equation*}
\left(\zeta_{\varepsilon}(x)\neq-\infty\wedge\zeta_{\varepsilon}(y)\neq-\infty\right)\implies\zeta_{\varepsilon}((xy)|_{\mathrm{int}})\geqslant\zeta_{\varepsilon}(x)+\zeta_{\varepsilon}(y)\text{.}
\end{equation*}
\end{prop}

\begin{proof}
We will first show that for any $(a,I),(b,J)\in\mathcal{P}$ such that $u(a)\neq0$, $I_{0}\neq\emptyset$ and $u(b)\neq0$, and any $\eta_{a,I},\eta_{b,J}\in W(k)$ we always have:
\begin{equation*}
\zeta_{\varepsilon}\left(\left(e(\eta_{a,I},a,I)e(\eta_{b,J},b,J)\right)|_{\mathrm{int}}\right)\geqslant\zeta_{\varepsilon}(e(\eta_{a,I},a,I))+\zeta_{\varepsilon}(e(\eta_{b,J},b,J))\text{.}
\end{equation*}

Due to proposition \ref{mainprop} we see that: 
\begin{multline*}
\zeta_{\varepsilon}\left(\left(e(\eta_{a,I},a,I)e(\eta_{b,J},b,J)\right)|_{\mathrm{int}}\right)\\\geqslant2n\left(\val_{V}(\eta_{a,I})+\val_{V}(\eta_{b,J})+\min\{u(a),u(b)\}\right)-\varepsilon\lvert a+b\rvert
\end{multline*}
because in the integral part we have $u(a+b)=0$, which is only possible if $u(a)=u(b)$. This proves this specific case because $2n\geqslant(\#I+1+\#J)$ if $J_{0}=\emptyset$, and $2n\geqslant(\#I+1+\#J+1)$ otherwise.

For the general case, notice that if $I_{0}=\emptyset$, then proposition \ref{dactionone} gives us the equality:
\begin{multline*}
e(\eta_{a,I},a,I)e(\eta_{b,J},b,J)=d(e(\eta_{a,I},a,I\smallsetminus\{\min(a)\})e(\eta_{b,J},b,J))\\-(-1)^{\#I-1}e(\eta_{a,I},a,I\smallsetminus\{\min(a)\})d(e(\eta_{b,J},b,J))\text{.}
\end{multline*}

Therefore, we can conclude using the first paragraph as well as \eqref{detzetavareps}.
\end{proof}

\begin{prop}
Let $\varepsilon>0$, $x\in W\Omega_{k[\underline{X}]/k}^{\mathrm{frp}}$ and $y\in d\left(W\Omega_{k[\underline{X}]/k}^{\mathrm{frp}}\right)$. We have:
\begin{equation*}
\left(\zeta_{\varepsilon}(x)\neq-\infty\wedge\zeta_{\varepsilon}(y)\neq-\infty\right)\implies\zeta_{\varepsilon}((xy)|_{\mathrm{d(frp)}})\geqslant\zeta_{\varepsilon}(x)+\zeta_{\varepsilon}(y)+1\text{.}
\end{equation*}
\end{prop}

\begin{proof}
We only have to show that for any $(a,I),(b,J)\in\mathcal{P}$ such that $u(a)\neq0$, $I_{0}\neq\emptyset$, $u(b)\neq0$ and $J_{0}=\emptyset$, and any Witt vectors $\eta_{a,I},\eta_{b,J}\in W(k)$ we have the inequality $\zeta_{\varepsilon}\left((e(\eta_{a,I},a,I)e(\eta_{b,J},b,J))|_{\mathrm{d(frp)}}\right)\geqslant\zeta_{\varepsilon}(e(\eta_{a,I},a,I))+\zeta_{\varepsilon}(e(\eta_{b,J},b,J))+1$. Using proposition \ref{mainprop} one finds that:
\begin{multline*}
\zeta_{\varepsilon}\left((e(\eta_{a,I},a,I)e(\eta_{b,J},b,J))|_{\mathrm{d(frp)}}\right)\\\geqslant2n\left(\val_{V}(\eta_{a,I})+\val_{V}(\eta_{b,J})+\min\{u(a),u(b)\}\right)\\+(\#I+\#J)u(a+b)-\varepsilon\lvert a+b\rvert\text{.}
\end{multline*}

So the proof is over if:
\begin{equation*}
2n\min\{u(a),u(b)\}+(\#I+\#J)u(a+b)\geqslant(\#I+1)u(a)+\#Ju(b)+1\text{.}
\end{equation*}

Since $\#I+1\leqslant n$ and $1\leqslant\#J\leqslant n$ because we have supposed that $I_{0}\neq\emptyset$ and $J_{0}=\emptyset$, and since $u(a+b)\neq0$ because we study the fractional part, this inequality becomes obvious whenever $u(a)=u(b)$; if not then $u(a+b)=\max\{u(a),u(b)\}$, and we are done.
\end{proof}

\begin{prop}
For any $\varepsilon>0$ and any $x,y\in d\left(W\Omega^{\mathrm{frp}}_{k[\underline{X}]/k}\right)$ we get:
\begin{equation*}
\left(\zeta_{\varepsilon}(x)\neq-\infty\wedge\zeta_{\varepsilon}(y)\neq-\infty\right)\implies\zeta_{\varepsilon}((xy)|_{\mathrm{d(frp)}})\geqslant\zeta_{\varepsilon}(x)+\zeta_{\varepsilon}(y)\text{.}
\end{equation*}
\end{prop}

\begin{proof}
One can suppose without any loss of generality that $x\in W\Omega_{k[\underline{X}]/k}^{i}$ for some $i\in\mathbb{N}$. Put $y'\in W\Omega_{k[\underline{X}]/k}^{\mathrm{frp}}$ such that $d(y')=y$. Using proposition \ref{dactionone} we get $\zeta_{\varepsilon}(y')=\zeta_{\varepsilon}(y)$. But $xy=(-1)^{i}d(xy')$, so we can conclude thanks to \eqref{zetaaddition}, \eqref{detzetavareps} as well as proposition \ref{frpdfrpfrpzeta}.
\end{proof}

Notice that if one takes $x,y\in W\Omega_{k[\underline{X}]/k}^{\mathrm{int}}$, then $xy\in W\Omega_{k[\underline{X}]/k}^{\mathrm{int}}$. In particular, $(xy)|_{\mathrm{frp}}=(xy)|_{\mathrm{d(frp)}}=0$, which implies that $\zeta_{\varepsilon}\left((xy)|_{\mathrm{frp}}\right)=\zeta_{\varepsilon}\left((xy)|_{\mathrm{d(frp)}}\right)=+\infty$ for any $\varepsilon>0$.

In a similar fashion, if $x\in W\Omega_{k[\underline{X}]/k}^{\mathrm{int}}$ and $y\in W\Omega_{k[\underline{X}]/k}^{\mathrm{frac}}$, lemma \ref{technicallem} implies that $\zeta_{\varepsilon}\left((xy)|_{\mathrm{int}}\right)=+\infty$ for any $\varepsilon>0$.

Also, if $x,y\in d\left(W\Omega_{k[\underline{X}]/k}^{\mathrm{frp}}\right)$, so as $xy$ lays in the image of $d$ we get $(xy)|_{\mathrm{frp}}=0$, which in turns implies that $\zeta_{\varepsilon}\left((xy)|_{\mathrm{frp}}\right)=+\infty$ for any $\varepsilon>0$.

The following table compiles all of the propositions we have shown concerning the function $\zeta_{\varepsilon}$ for any $\varepsilon>0$ (we will always suppose that $\zeta_{\varepsilon}(x)\neq-\infty$ and $\zeta_{\varepsilon}(y)\neq-\infty$).
\begin{equation*}
\begin{array}{|c|c|c|c|}
\hline&\zeta_{\varepsilon}\left((xy)|_{\mathrm{int}}\right)\geqslant&\zeta_{\varepsilon}\left((xy)|_{\mathrm{frp}}\right)\geqslant&\zeta_{\varepsilon}\left((xy)|_{\mathrm{d(frp)}}\right)\geqslant\\
\hline\begin{array}{c}x\in W\Omega_{k[\underline{X}]/k}^{\mathrm{int}}\\y\in W\Omega_{k[\underline{X}]/k}^{\mathrm{int}}\end{array}&\zeta_{\varepsilon}\left(x\right)+\zeta_{\varepsilon}\left(y\right)&+\infty&+\infty\\
\hline\begin{array}{c}x\in W\Omega_{k[\underline{X}]/k}^{\mathrm{frp}}\\y\in W\Omega_{k[\underline{X}]/k}^{\mathrm{int}}\end{array}&+\infty&\zeta_{\varepsilon}\left(x\right)+\zeta_{\varepsilon}\left(y\right)&\zeta_{\varepsilon}\left(x\right)+\zeta_{\varepsilon}\left(y\right)+1\\
\hline\begin{array}{c}x\in d\left(W\Omega_{k[\underline{X}]/k}^{\mathrm{frp}}\right)\\y\in W\Omega_{k[\underline{X}]/k}^{\mathrm{int}}\end{array}&+\infty&\zeta_{\varepsilon}\left(x\right)+\zeta_{\varepsilon}\left(y\right)&\zeta_{\varepsilon}\left(x\right)+\zeta_{\varepsilon}\left(y\right)\\
\hline\begin{array}{c}x\in W\Omega_{k[\underline{X}]/k}^{\mathrm{frp}}\\y\in W\Omega_{k[\underline{X}]/k}^{\mathrm{frp}}\end{array}&\zeta_{\varepsilon}\left(x\right)+\zeta_{\varepsilon}\left(y\right)+2&\zeta_{\varepsilon}\left(x\right)+\zeta_{\varepsilon}\left(y\right)+1&\zeta_{\varepsilon}\left(x\right)+\zeta_{\varepsilon}\left(y\right)+3\\
\hline\begin{array}{c}x\in d\left(W\Omega_{k[\underline{X}]/k}^{\mathrm{frp}}\right)\\y\in W\Omega_{k[\underline{X}]/k}^{\mathrm{frp}}\end{array}&\zeta_{\varepsilon}\left(x\right)+\zeta_{\varepsilon}\left(y\right)&\zeta_{\varepsilon}\left(x\right)+\zeta_{\varepsilon}\left(y\right)&\zeta_{\varepsilon}\left(x\right)+\zeta_{\varepsilon}\left(y\right)+1\\
\hline\begin{array}{c}x\in d\left(W\Omega_{k[\underline{X}]/k}^{\mathrm{frp}}\right)\\y\in d\left(W\Omega_{k[\underline{X}]/k}^{\mathrm{frp}}\right)\end{array}&\zeta_{\varepsilon}\left(x\right)+\zeta_{\varepsilon}\left(y\right)&+\infty&\zeta_{\varepsilon}\left(x\right)+\zeta_{\varepsilon}\left(y\right)\\
\hline
\end{array}
\end{equation*}

In particular, this proves the main theorem of this paper:

\begin{thrm}
For any $\varepsilon>0$, the function $\zeta_{\varepsilon}$ is a pseudovaluation.
\end{thrm}

\begin{proof}
This is straightforward using \eqref{zetaaddition} and the previous table.
\end{proof}

In subsequent papers, we will use this theorem and this table in order to study the local structure of the overconvergent de Rham-Witt complex, and give an interpretation of $F$-isocrystals in this context.

\end{document}